\documentclass{amsart}

\usepackage{soul,color}
\newtheorem{theorem}{Theorem}[section]
\newtheorem{corollary}[theorem]{Corollary}
\newtheorem{lemma}[theorem]{Lemma}

\numberwithin{equation}{section}

\def\EE{\text{\pe E}}
\def\var{\text{\pe Var}}

\def\asarrow{\buildrel a.s.\over\longrightarrow}
\def\varrow{\buildrel v\over\longrightarrow}
\def\tr{\text{\pe tr}\,}
\font\pe=cmss10

\def\comp{\Bbb C}

\def\bfone{\text{\bf 1}}

\def\P{\text{\pe P}}
\def\xn{x^{(n)}}
\def\vertiii#1{{\left\vert\kern-0.25ex\left\vert\kern-0.25ex\left\vert #1
    \right\vert\kern-0.25ex\right\vert\kern-0.25ex\right\vert}}

\newcommand{\comment}[1]{{\color{blue}\fbox{#1}}}
\newcommand{\longcomment}[1]{\fbox{\begin{minipage}{.9\textwidth}\color{blue}#1\end{minipage}}}

\title[Limiting Eigenvalue Behavior of Random Matrices]{Limiting Eigenvalue Behavior of a Class of Large Dimensional Random Matrices Formed From a Hadamard Product}
\author{Jack W. Silverstein}

\address{Jack W. Silverstein,
Department of Mathematics, Box 8205, North Carolina State University, Raleigh, NC 27695-8205, USA
}
\email{jack@ncsu.edu}


\begin{document}

\begin{abstract} 
This paper investigates the strong limiting behavior of the eigenvalues of the class of matrices 
$\frac1N(D_n\circ X_n)(D_n\circ X_n)^*$, studied in Girko 2001.  Here, $X_n=(x_{ij})$ is an $n\times N$ random
matrix consisting of independent complex standardized random variables, $D_n=(d_{ij})$, $n\times N$, has
nonnegative entries, and $\circ$ denotes Hadamard (componentwise) product.  Results are obtained under
assumptions on the entries of $X_n$ and $D_n$ which are different from those in Girko (2001), which include a
Lindeberg condition on the entries of $D_n\circ X_n$, as well as a bound on the average of the rows and columns 
of $D_n\circ D_n$.  The present paper separates the assumptions needed on $X_n$ and $D_n$. It assumes a 
Lindeberg condition on the entries of $X_n$, along with a tigntness-like condition on the entries of $D_n$, 
\end{abstract}

\maketitle
\section{Introduction}
This paper deals with the limiting eigenvalue behavior of the class of Hermitian nonnegative definite matrices
\begin{equation}\label{Bndef}B_n=\frac1N(D_n\circ X_n)(D_n\circ X_n)^*,\end{equation}
where for each positive integer $n$ $X_n=(x^{(n)}_{ij})$ is $n\times N$ with random variables $x^{(n)}_{ij}\in\Bbb C$, independent, and standardized ($\EE x^{(n)}_{ij}=0$, $\EE|x^{(n)}_{ij}|^2=1$), $D_n$ is
$n\times N$ containing nonrandom, nonnegative real numbers $d_{ij}=d^n_{ij}$, $\circ$ denotes Hadamard product, and $N=N(n)$ with $0<\liminf_n n/N\leq\limsup_n n/N<\infty$.  Such matrices arise in various situations when the dimension is large and where there is no prescribed structure to the elements in the $D_n$ matrix.

A standard way to pursue the limiting eigenvalue behavior of Hermitian random matrices as the dimension increases is through the empirical distribution 
function (e.d.f) of their eigenvalues, that is, for random Hermitian $n\times n$ matrix $A_n$, let for every $x\in\Bbb R$, $F^{A_n}(x)=\{\text{number of eigenvalues
of $A_n$ $\leq x$}\}/n.$  A standard tool used in understanding the (e.d.f.) of the eigenvalues has been since Mar\v cenko Pastur (1967),  the Stieltjes transform, 
where for arbitrary finite measure $\mu$ on $\Bbb R$ is defined by
\begin{equation}\label{STdef}m_{\mu}=\int\frac1{x-z}d\mu(x),\quad z\in\Bbb C^+\equiv\{z\in\Bbb C:\Im>0\}.
\end{equation}
It is analytic and takes values in  $\Bbb C^+$.  Notice that $m_{\mu}(z)|\leq\mu(\Bbb R)/\Im z$.
The Stieltjes transform of the measure induced by $F^{A_n}$ is then
$$m_{A_n}=\int\frac1{x-z}dF^{A_n}(x)=\frac1n\tr(A_n-zI)^{-1},$$
where  $I$ is the $n\times n$ identity matrix. and $\tr$ is the trace of a matrix. 

Due to the inversion formula
\begin{equation}\label{STinv}\mu([a,b])=\frac1{\pi}\lim_{\eta\to 0^+}\int_a^b\Im m_{\mu}(\xi+i\eta)d\xi\quad\text{($a$, $b$ continuity points of $\mu$)},
\end{equation}
it will follow that understanding the limiting behavior of $F^{A_n}$ can be handled by its Stieltjes transform.  

Work on the eigenvalue behavior of $B_n$ has been done in [Girko 2001].  Indeed,
  Theorem 10.1 of [Girko 2001], contains the following result:  assume  $x_{ij}^{(n)}\in\Bbb R$,  
\begin{equation}\label{davbd}\sup_n \max_{\substack{i=1,\ldots,n\\j=1,\ldots,N}}\biggl\{\frac1n\sum_{i=1}^N d_{ij}^2+\frac1N\sum_{j=1}^N d_{ij}^2\biggr\}<\infty,
\end{equation}
and a Lindeberg condition is satisfied, namely for arbitrary $\eta>0$
\begin{multline*}
\lim_{n\to\infty} \max_{\substack{ i=1,\ldots,n\\j=1,\ldots,N}}\biggl\{\frac1n\sum_{i=1}^nd_{ij}^2\EE((x_{ij}^{(n)})^2I(d_{ij}|x_{ij}^{(n)}|>\eta\sqrt n))
\hfill \\ \hfill+
\frac1N\sum_{j=1}^Nd_{ij}^2\EE((x_{ij}^{(n)})^2I(d_{ij}|x_{ij}^{(n)}|>\eta\sqrt n))\biggr\}=0,
\end{multline*}
where $I(A)$ is the indicator function on the set $A$.
Then, with $\|\cdot\|$ denoting the sup norm on functions, for each $n$, there exists a nonrandom probability distribution function $F^0_n$, such that, with probability one
$$\lim_{n\to\infty}\|F^{B_n}-F_n^0\|=0.$$
and $F_n^0$ has Stieltjes transform
\begin{equation}\label{GST}G_n(z):=\frac1n{   \sum_{i=1}^n}\frac1{ \frac1N\sum_{k=1}^N\frac{d_{ik}^2}{1+\frac nNe_k^0(z)}-z},
\end{equation}
where for each $z\in\Bbb C^+$, $e_1^0,\ldots,e_N^0$ are unique solutions lying in $\Bbb C^+$ to the system of equations
\begin{equation}
  \label{Ge0def}
  e_j^0(z)=\frac1n\sum_{i=1}^n\frac{d_{ij}^2}{ \frac1N\sum_{k=1}^N\frac{d_{ik}^2}{1+\frac nNe^0_k(z)}-z}.
\end{equation}

This result is among a collection which differs from typical results on limiting eigenvalue behavior (as the dimension of the matrix increases), in that there is no statement on what a possible limiting
e.d.f of the eigenvalues of $B_n$ could be.    However, the result is important in that it shows that  $F^{B_n}$ is becoming less random,  with 
a way  for deriving, through \eqref{STinv}, what it is close to.   The $F_n^0$'s can be thought of as  {\sl deterministic equivalents} of the e.d.f.'s

The aim of the current paper is to prove a result under different assumptions, in particular, for each $\eta>0$
\begin{equation}\label{Lind}\lim_{n\to\infty}\frac1{nN}\sum_{jk}\EE|\xn_{jk}|^2I(|\xn_{jk}|\ge\eta\sqrt n)=0,
\end{equation}
From this it is straightforward to construct a sequence $\{\eta_n\}$ of positive numbers for which $\eta_n\downarrow0$ as $n\to\infty$ and
\begin{equation}
  \label{1.2}
\lim_{n\to\infty}\frac1{\eta_n^2nN}\sum_{jk}\EE|\xn_{jk}|^2I(|\xn_{jk}|\ge\eta_n\sqrt n)=0.
\end{equation}

We also assume the existence of  positive $e$ and $f$ such that for all $n$ sufficiently large
\begin{equation}
  \label{1.3}
e<\eta_n\sqrt n\quad\text{and}\quad\EE\left|\xn_{jk}I(|\xn_{jk}|\leq e)-\EE( \xn_{jk}I(|\xn_{jk}|\leq e))\right|^2>f
\end{equation}
Also, we assume each column of $D_n$ is nonzero, and the matrix satisfies the following property.   For every $\epsilon>0$ there exists an $M_{\epsilon}>0$ such that for each $n$ there exists sets $E^n_{r\epsilon}\subset\{1,2,\ldots,n\},\ E^n_{c\epsilon}\subset\{1,2,\ldots,N\}$ such that
\begin{enumerate}
\item $\#E^n_{r\epsilon}+\#E^n_{c\epsilon}\leq \epsilon n$ ($\#$ denotes number of elements in the set)
\item $d_{jk}^n\leq M_{\epsilon}$ for $j\in {E^n_{r\epsilon}}^c$ and $k\in {E^n_{c\epsilon}}^c$.
\end{enumerate}

The result is expressed in terms of the following metric on sub-probability measures on $\Bbb R$ via their distribution functions:
$$D(F,G)\equiv\sum_{i=1}^{\infty}\left|\int f_idF-\int f_idG\right|2^{-i},$$
where $\{f_i\}$ is an enumeration of all continuous functions that take a constant $\frac1m$
value ($m$ a positive integer) on $[a,b]$ where $a,b$ are rational, 0 on $(-\infty,a-\frac1m]\cup[b+\frac1m,\infty)$, and linear on each of $[a-\frac1m,a],[b,b+\frac1m]$.  It is straightforward to verify that $D(\cdot,\cdot)$ induces the topology of weak convergence on probability measures, and the topology of vague convergence on the set of sub-probability measures.    Since for $x,y\in\Bbb R$, $|f_i(x)-f_i(y)|\leq|x-y|$, we have for two empirical distribution functions (e.d.f.) $F,G$ on the respective sets $\{x_1,\ldots,x_n)\},\{y_1,\ldots,y_n\}$
\begin{equation}
  \label{1.1}
  D(F,G)\leq\frac1n\sum_{j=1}^n|x_j-y_j|\leq \left(\frac1n\sum_{j=1}^n(x_j-y_j)^2\right)^{1/2}. 
\end{equation}

We will prove 
\begin{theorem}\label{1.1} For each $\epsilon>0$ choose $d_{\epsilon}\ge M_{\epsilon}$ and define $\widetilde d^n_{jk}=\widetilde d^{n\epsilon}_{jk}=d^n_{jk}I(d^n_{jk}\leq d_{\epsilon})$.  Then with probability one
\begin{equation}\label{epresult}
\limsup_nD(F^{B_n},F^0_{n,\epsilon})\leq\epsilon,
\end{equation}
where $F^0_{n,\epsilon}$ is the probability distribution function having Stieltjes transform
\begin{equation}\label{epST}G_n(z)=\frac1n{   \sum_{i=1}^n}\frac1{ \frac1N\sum_{k=1}^N
\frac{\widetilde d^{n\!\ 2}_{ik}}{1+\frac nNe_k^0(z)}-z},
\end{equation}
and for each $z\in\Bbb C^+$, $e_k^0=e^0_{n,k}(z)$ are unique solutions lying in $\Bbb C^+$ to the system of equations
\begin{equation}
  \label{epe0def}
  e_j^0(z)=\frac1n\sum_{i=1}^n\frac{\widetilde d_{ij}^{n\!\ 2}}{ \frac1N\sum_{k=1}^N\frac{{\widetilde d_{ik}}^{n\!\ 2}}{1+\frac nNe^0_k(z)}-z}.
\end{equation}
\end{theorem}

\begin{corollary}With probability one, there exists a (random) sequence $\{\epsilon_n\}$, such that, with $\widetilde d^n_{ij}$ and $F^0_{n,\epsilon_n}$ defined as above, we have
\begin{equation}\label{epnlimit} \lim_{n\to\infty}D(F^{B_n},F^0_{n,\epsilon_n})=0.
\end{equation}
\end{corollary}

The present assumptions allow for a clearer understanding as to the conditions needed on the entries of $X_n$ and $D_n$ separately, so for example, determining the 
applicability of the theorem on matrix ensembles with the same $X_n$ satisfying \eqref{Lind},\eqref{1.2},\eqref{1.3} would only require investigating the properties on different $D_n$'s.
  The Lindeberg condition is just assumed
on the $x^{(n)}_{ij}$'s, while a tightness-like condition on the $d_{ij}$'s is only needed.  There can be some exceedingly large values of $d_{ij}$, just not too many of them.   The
value on the left side of \eqref{davbd} can go unbounded.

The conclusion of the Theorem should not be considered substantially weaker than the one in Theorem 10.1 of [Girko 2001].   Typical of eigenvalue results of this nature, it is not
known how large the dimension should be in order to obtain prescribed accuracy and reliability.   One of the main uses of these limit laws is to be able to identify and understand the
underlying assumptions on the makeup of the matrix, typically by viewing a histogram of the random eigenvalues.   

The proofs of the Theorem and corollary are given in section 3, following a section on lemmas needed in the proof.  Section 3 ends with a scheme to compute $e^0$.

\section{Lemmas}
This section contains results on matrix theory and probability need in the proofs.
\begin{lemma} [{\cite[Corollary 7.3.8]{horn1990matrix}}]
  \label{L.1.1} 
  For $r\times s$ matrices $A$ and $B$ with respective singular values $\sigma_1\ge\cdots\ge\sigma_q,\ \tau_1\ge\cdots\ge\tau_q$ where $q=\min\{r,s\}$, we have
$$\left(\sum_{i=1}^q(\sigma_i-\tau_i)^2\right)^{1/2}\leq\|A-B\|_2,$$
where $\|\cdot\|_2$ is the Frobenius matrix norm.
\end{lemma}

Let $\|\cdot\|$ denote the sup-norm on bounded functions from $\Bbb R$ to $\Bbb R$.  It is clear that for probability distribution 
functions $F$ and $G$
$$D(F,G)\leq \|F-G\|.$$

\begin{lemma}
\label{L.1.1.5}
For matrices $A$, $B$ of the same dimension, $rank(A+B)\leq rank(A) + rank(B)$.
\end{lemma}

\begin{lemma}
\label{Lrowcol}
Let $A$ be $r\times s$.  Then $rank(A)\leq$ number of nonzero rows of $A$ $+$ the number of nonzero columns of $A$.
\end{lemma}
\begin{proof} Let $e^i_m\in\Bbb R^m$ be the canonical vector with 1 in the $i-th$ position and zero in the remaining positions, $a_{j\cdot}$ the $j-th$ row of $A$, 
and $a_{\cdot k}$ the $k-th$ column of $A$. Then we can write 
$$A=\sum_{j=1}^re^j_ra_{j\cdot}=\sum_{k=1}^sa_{\cdot k}{e^k_s}^T=\frac12\sum_{j=1}^re^j_ra_{j\cdot}+\frac12\sum_{k=1}^sa_{\cdot k}{e^k_s}^T.$$
Each of the terms in the sums is a $rank$ 1 matrix.  Removing all the zero rows and columns and using Lemma \ref{L.1.1.5} we get our result.
\end{proof}

For Hermitian matrix $A$ we let $F^A$ denote the e.d.f. of the eigenvalues of $A$, and for rectangular matrix $B$ we let $F^B_{sing}$ denote the e.d.f. of the singular values of $B$.

\begin{lemma}[{\cite[Theorem A.44]{bai2010spectral}}]
   \label{L.1.2}
   For $r\times s$ matrices $A$ and $B$
$$\|F^{AA^*}-F^{BB^*}\|\leq \frac1r\text{rank}(A-B).$$
\end{lemma}

Since the rank of a matrix is bounded above by the number of its non-zero rows, we have

\begin{lemma}
  \label{L.1.3}
  For matrices in Lemma \ref{L.1.2}
\begin{multline*}\|F^{AA^*}-F^{BB^*}\|\leq \frac1r\{\text{number of non-zero rows of $A-B$}\}\hfill\\ \hfill\leq\frac1r\{\text{number of non-zero entries of $A-B$}\}.\end{multline*}
\end{lemma}

\begin{lemma}[Bernstein's inequality]
  \label{L.1.4}
  Let  $X_1,\ldots,X_n$ denote independent mean zero random variables uniformly bounded in absolute value by b, $S_n=X_1+\ldots+X_n$ and $\sigma_n^2=\EE S_n^2$.  Then for any $\epsilon>0$
$$\P(S_n\ge \epsilon)\leq \exp \left(\frac{-\epsilon^2}{2(\sigma_n^2+\frac{b\epsilon}3)}\right).$$
\end{lemma}

\begin{lemma}[consequence of Burkholder's inequality]
  \label{L.1.5}
  For $\{X_k\}$ independent mean-zero random variables we have for any $p\ge2$
$$\EE\bigl|\sum_kX_k\bigr|^p\leq C_p\bigl(\sum_k \EE|X_k|^p+\bigl(\sum_k\EE|X_k|^2\bigr)^{p/2}\bigr).$$

\end{lemma}

\begin{lemma}[\cite{bai1998no}]
  \label{L.1.6}
   Let $\{X_k\}$ be a complex martingale difference sequence.  Then, for $p>1$
$$\EE\biggl|\sum X_k\biggr|^p\leq C_p\EE\biggl(\sum|X_k|^2\biggr)^{p/2}.$$
\end{lemma}

\begin{lemma}
  \label{L.1.7}
  Let $B,C$ be $n\times n$ with $B$ Hermitian, $x\in\Bbb C^n$ and $z=x+iv$ with $v>0$.  Then
$$\frac1{|z(1+x^*(B-zI)^{-1}x)|}\leq\frac1v\quad\text{and}\quad\frac1{|z(1+\tr C^*(B-zI)^{-1}C)|}\leq\frac1v.$$

\end{lemma}
\begin{proof}
 Follows from the fact that the imaginary parts of $x^*((1/z)B-I)^{-1}x$ and $\tr C^*((1/z)B-I)^{-1}C$ are both non-negative.
\end{proof}

\begin{lemma}[\cite{bai1998no}]
  \label{L.1.8}
  Let $z\in\Bbb C$ with $v=\Im z>0$, $A$, $B$ $n\times n$ with $B$ Hermitian, and $r\in\Bbb C^n$.  Then
$$|\tr((B-zI)^{-1}-(B+rr^*-zI)^{-1})A|=\left|\frac{r^*(B-zI)^{-1}A(B-zI)^{-1}r}{1+r^*(B-zI)^{-1}r}\right|\leq\frac{\vertiii{A}_2}v,$$
where $\vertiii{\cdot}_2$ denotes spectral norm.

\end{lemma}
\begin{lemma}
  \label{L.1.9}
  For $A$, $r$ as in Lemma \ref{L.1.8} with $A$ and $A+rr^*$ both invertible, we have
$$r^*(A+rr^*)^{-1}=\frac1{1+r^*A^{-1}r}r^*A^{-1}.$$
Proof.  Follows from $r^*A^{-1}(A+rr^*)=(1+r^*A^{-1}r)r^*$.
\end{lemma}

\begin{lemma}[{\cite[Lemma B.26]{bai2010spectral}}]
  \label{L.1.10}
  Let $A$ be $n\times n$ and $x=(x_1,\ldots,x_n)^T$ where the $x_i$ are independent random variables with $\EE x_i=0$, $\EE|x_i|^2=1$, and $\EE|x_i|^{\ell}\leq\nu_{\ell}$.  Then for any $p\ge1$
$$\EE|x^*Ax-\tr A|^p\leq C_p\left((\nu_4\tr(AA^*))^{p/2}+\nu_{2p}\tr(AA^*)^{p/2}\right).$$
\end{lemma}

\begin{lemma}
[{\cite[Theorem 8.3.1]{horn1990matrix}}]
   \label{L.1.11}
   Let $\rho(C)$ denote the spectral radius of the $N\times N$ matrix $C$ (the largest of the absolute values of the eigenvalues of C).   If $C$ contains only nonnegative entries, then $\rho(C)$ is an eigenvalue of $C$ having an eigenvector with nonnegative entries.
\end{lemma}
\begin{lemma}
[{\cite[Theorem 8.1.18]{horn1990matrix}}]
  \label{L.1.12}
  Suppose $A=(a_{ij})$ and $B=(b_{ij})$ are $N\times N$ with $b_{ij}$ nonnegative and $|a_{ij}|\leq b_{ij}$. Then
$$\rho(A)\leq\rho((|a_{ij}|))\leq\rho(B).$$
\end{lemma}

\begin{lemma}
[{\cite[Lemma 5.7.9]{horn1991topics}}]
   \label{L.1.13}
   Let $A=(a_{ij})$ and $B=(b_{ij})$ be $N\times N$ with $a_{ij}$, $b_{ij}$ nonnegative. Then
$$\rho((a_{ij}^{\frac12}b_{ij}^{\frac12}))\leq(\rho(A))^{\frac12}(\rho(B))^{\frac12}.$$
\end{lemma}

\begin{lemma}[{\cite[Lemma 5.6.10]{horn1990matrix}}]
  \label{L.1.14}
  For square $A$ and $\epsilon >0$ there exists a matrix norm $\vertiii{\cdot}$ such that $\rho(A)\leq\vertiii{A}\leq\rho(A)+\epsilon$.
\end{lemma}

\begin{lemma}
[{\cite[Theorem 5.6.26]{horn1990matrix}}]
  \label{L.1.15}
  Let $\vertiii{\cdot}$ be a given matrix norm on $n\times n$ matrices.  Then there exists an induced matrix norm $N(\cdot)$ on $n\times n$ matrices such that, for any $n\times n$   matrix $A$  we have
  $N(A)\leq\vertiii{A}$.

\end{lemma}
From the last two lemmas, we have
\begin{lemma}
  \label{L.1.16}
  Let $A$ be $n\times n$ satisfy $\rho(A)<1$.   Then there exists a vector norm  on $\Bbb C^n$, such that with $\vertiii{\cdot}$ denoting its induced matrix norm, we have $\vertiii{A}<1$.

\end{lemma}



\begin{lemma}\label{supnormdiag} Let $A$ be an $m\times n$ matrix.  Then $\vertiii{A}_2=\vertiii{A\text{diag}(\omega_1,\ldots,\omega_n)}_2$, where the $\omega_i$ are numbers on the unit circle in the complex plane.

Proof. Let $y_i\in\Bbb C^n$, $i=1,2$ be unit vectors for which $\vertiii{A}=\|Ay_1\|_2$ and 
$\vertiii{A\text{diag}(\omega_1,\ldots,\omega_n)}_2=\|A\text{diag}(\omega_1,\ldots,\omega_n)y_2\|_2$.  Then
\begin{multline*}\vertiii{A}_2\ge\|A\text{diag}(\omega_1,\ldots,\omega_n)y_2\|_2=\vertiii{A\text{diag}(\omega_1,\ldots,\omega_n)}_2\hfill\\ \hfill\ge\|A\text{diag}(\omega_1,\ldots,\omega_n)
\text{diag}(\overline\omega_1,\ldots,\overline\omega_n)y_1\|_2=\vertiii{A}_2,\end{multline*}
so we get our result.
\end{lemma}

For sub-probability measures $\{\mu_n\}$, $\mu$, since for fixed $z\in\Bbb C^+$ the real and imaginary parts of $1/(x-z)$ are continuous and approach 0 as $|x|\to\infty$, we have 
$\mu_n\varrow\mu$ ($\varrow$ denoting vague convergence) implies $m_{\mu_n}(z)\to m_{\mu}(z)$.  Conversely, if $m_{\mu_n}(z)$ converges for a countably infinite number of $z\in\Bbb C^+$ possessing a cluster point, all uniformly bounded away from the real axis, from Vitali's convergence theorem  [{\cite[p. 168]{titchmarsh}}], $m_{\mu_n}(z)$ converges for all $z$ uniformly bounded away from the real axis to an analytic function $m$.   
Therefore any vaguely converging subsequence of $\{m_{\mu}\}$ has their Stieltjes transforms converging to $m$, and  because of the existence of the inverse formula \eqref{STinv} we see that 
$\mu_n$ converges vaguely to a sub-probability measure $\mu$ having Stieltjes transform $m$.  Thus we have

\begin{lemma}\label{vaguelimit} If for sub-probability measures $\mu_n$, we have $m_{\mu_n}(z)$ converging for a countably infinite number of $z$ uniformly bounded away from the real axis and
possessing a cluster point, 
then there exists a sub-probability measure $\mu$ for which $\mu_n\varrow\mu$, or equivalently $D(\mu_n,\mu)\to0$.
\end{lemma}

\begin{lemma} [{\cite[Lemma 2.2]{shohat}}]
 \label{stielt} Let $f$ be analytic in $\Bbb C^+$ mapping $\Bbb C^+$ into $\Bbb C^+$, and there is a $\theta\in(0,\pi/2)$ for which $zf(z)\to c$, finite, as $z\to\infty$ restricted to $\{w\in\Bbb C^+:\theta<\arg w<\pi-\theta\}$.  Then $f$ is the Stieltjes transform of a measure with total mass $-c$.
\end{lemma}

\begin{lemma}[{\cite[Corolllary 8.1.20]{horn1990matrix}}]\label{diagboun} Let $C$ $N\times N$ have nonnegative entries.  Then for each $i$ $C_{ii}\leq\rho(C)$.
\end{lemma}

\begin{lemma}[{\cite[Theorem 2.1]{sandc}}]\label{density} Let $G$ be a probability distribution function and $x_0\in\Bbb R$. Let $m_G$ be its Stieltjes transform.  Suppose $\Im m_G(x_0)\equiv\lim_{z\in\Bbb C^+}\Im m_G(z)$ exists.  Then $G$ is differentiable at $x_o$, and its derivative is $\frac1{\pi}\Im m_G(x_0)$.
\end{lemma}

\section{Proofs of the Theorem and Corollary}
We begin by performing a series of truncations and centralizations on the entries of $X_n$ and a truncation on the entries of $D_n$

Let
$$\widetilde B_n=\frac1N(D_n\circ\widetilde X_n)(D_n\circ\widetilde X_n)^*,$$
where
$$\hbox{$\widetilde X_n$}_{jk}=\xn_{jk}I(|\xn_{jk}|\leq\eta_n\sqrt n).$$

Then from Lemma \ref{L.1.3}

$$\|F^{B_n}-F^{\widetilde
 B_n}\|\leq\frac1n\sum_{jk}I(|\xn_{jk}|\ge\eta_n\sqrt n).$$
We have by   \eqref{1.2}
$$\EE\bigg(\frac1n\sum_{jk}I(|\xn_{jk}|\ge\eta_n\sqrt n)\biggr)\leq\frac1{\eta_n^2n^2}\sum_{jk}\EE|\xn_{jk}|^2I(|\xn_{jk}|\ge\eta_n\sqrt n)=o(1),$$
and
$$\var\biggl(\frac1n\sum_{jk}I(|\xn_{jk}|\ge\eta_n\sqrt n)\biggr)\leq\frac1{\eta_n^2n^3}\sum_{jk}\EE|\xn_{jk}|^2I(|\xn_{jk}|\ge\eta_n\sqrt n)=o(1/n).$$
Therefore, from Lemma \ref{L.1.4}, for arbitrary positive $\epsilon$ we have for all $n$ large
\begin{multline*}
\P\biggl(\frac1n\sum_{jk}I(|\xn_{jk}|\ge\eta_n\sqrt n)>\epsilon)\leq\exp \left(\frac{-(\epsilon-o(1))^2}{2(o(1/n)+\frac{(1/n)(\epsilon-o(1))}3)}\right)\hfill\\ \hfill
\leq \exp\left(\frac{-n\epsilon^2}{8(1+\epsilon/3)}\right),
\end{multline*}
which is summable.  Therefore, $\|F^{B_n}-F^{\widetilde B_n}\|\asarrow0$ as $n\to\infty$.



For fixed $\epsilon>0$, define $d=d_{\epsilon}\ge M_{\epsilon}$ and  $\widetilde d^n_{jk}=\widetilde d^{n\epsilon}_{jk}$ as in the statement of Theorem 1.1, and let $\widetilde D_n=(\widetilde d^n_{jk})$.
From Lemma \ref{L.1.2} we have
$$\|F^{\widetilde B_n}-F^{(1/N)(\widetilde D_n\circ \widetilde X_n)(\widetilde D_n\circ \widetilde X_n)^*}\|
\leq\frac1n rank(\widetilde X_n\circ(D_n-\widetilde D_n))$$
The matrix $\widetilde X_n\circ(D_n-\widetilde D_n)$ has at most $\#E^n_{r\epsilon}$ nonzero rows and  $\#E^n_{c\epsilon}$ columns. Therefore from Lemma \ref{Lrowcol} we have
$$\|F^{\widetilde B_n}-F^{(1/N)(\widetilde D_n\circ \widetilde X_n)(\widetilde D_n\circ \widetilde X_n)^*}\|\leq\epsilon.$$


By    \eqref{1.1} and Lemma \ref{L.1.1} we have
\begin{multline*} D^2\biggl(F^{(1/\sqrt N)\widetilde D_n\circ\widetilde X_n}_{sing},F^{(1/\sqrt N)\widetilde D_n\circ(\widetilde X_n-\EE\widetilde X_n)}_{sing}\biggr)\leq d^2\frac1{nN}\sum_{jk}|\EE\xn_{jk}I(|\xn_{jk}
\leq\eta_n\sqrt n)|^2\hfill\\ \hfill =d^2\frac1{nN}\sum_{jk}|\EE\xn_{jk}I(|\xn_{jk}>\eta_n\sqrt n)|^2  \leq d^2\frac1{nN}\sum_{jk}\EE(|\xn_{jk}|^2I(|\xn_{jk}|\ge\eta_n\sqrt n)\to0\end{multline*}
by   \eqref{1.2}.
 Since  the set of subprobability measures is sequentially compact in vague topology with metric $D$ and taking square roots of   non-negative random variables is a continuous function, we have
$D(F^{\widetilde B_n},F^{(1/N)(\widetilde D_n\circ (\widetilde X_n-\EE\widetilde X_n))(\widetilde D_n\circ (\widetilde X_n-\EE\widetilde X_n))^*})\to0$ as $n\to\infty$.

Let
$$\widehat X_n=\left(\frac{\xn_{jk}I(|\xn_{jk}|\leq\eta_n\sqrt n) -\EE(\xn_{jk}I(|\xn_{jk}|\leq\eta_n\sqrt n))}{\sigma_{jk}}\right),$$
where $\sigma^2_{jk}=\sigma^{2(n)}_{jk}=\EE|\xn_{jk}I(|\xn_{jk}|\leq\eta_n\sqrt n)-\EE(\xn_{jk}I(|\xn_{jk}|\leq\eta_n\sqrt n))|^2$ (if $\sigma_{jk}=0$, then define the corresponding entry of $\widehat X_n$ to be zero).  Notice that $\sigma_{jk}\leq1$ and, by
  \eqref{1.3}, is $>\sqrt f$ for all $n$ large.
Then, again, by  \eqref{1.1} and Lemma \ref{L.1.1} we have
\begin{multline*} D^2(\biggl(F^{(1/\sqrt N)\widetilde D_n\circ\widehat X_n}_{sing},F^{(1/\sqrt N)\widetilde D_n\circ(\widetilde X_n-\EE\widetilde X_n)}_{sing}\biggr)\hfill\\ \hfill\leq d^2\frac1{nN}\sum_{jk}(1-\sigma_{jk}^{-1})^2|\xn_{jk}I(|\xn_{jk}|\leq\eta_n\sqrt n) -\EE(\xn_{jk}I(|\xn_{jk}|\leq\eta_n\sqrt n))|^2\equiv d^2a(n).\end{multline*}
We have
\begin{multline*}\EE a(n)=\frac{1}{nN}\sum_{jk}(1-\sigma_{jk})^2\leq \frac{1}{nN}\sum_{jk}(1-\sigma_{jk}^2)\hfill\\ \hfill=\frac{1}{nN}\sum_{jk}\EE|\xn_{jk}|^2I(\xn_{jk}|>\eta_n\sqrt n)+|\EE \xn_{jk}I(|\xn_{jk}>\eta_n\sqrt n)|^2\hfill\\ \hfill \leq \frac{2}{nN}\sum_{jk}\EE|\xn_{jk}|^2I(\xn_{jk}|\ge\eta_n\sqrt n)\to0,\end{multline*}
as $n\to\infty$, by  (2.1).

Let $a_{jk}=|\xn_{jk}I(|\xn_{jk}|\leq\eta_n\sqrt n) -\EE(\xn_{jk}I(|\xn_{jk}|\leq\eta_n\sqrt n))|^2$.  Using Lemma \ref{L.1.5} and   \eqref{1.3} for all $n$ large we have
\begin{multline*}\EE|a(n)-\EE a(n)|^4\hfill\\ \hfill\leq \frac{C_4}{(nN)^4}\biggr(\sum_{jk}(1-\sigma_{jk}^{-1}))^8\EE|a_{jk}-\EE a_{jk}|^4+\biggl(\sum_{jk}(1-\sigma_{jk}^{-1}))^4\EE|a_{jk}-\EE a_{jk}|^2\biggr)^2\biggr)\hfill\\
\hfill\leq\frac{C'}{n^8}\biggl(\sum_{jk}\EE|\xn_{jk}|^8I(|\xn_{jk}|\leq\eta_n\sqrt n)+\biggl(\sum_{jk}\EE|\xn_{jk}|^4I(|\xn_{jk}|\leq\eta_n\sqrt n)\biggr)^2\biggr)\hfill\\ \hfill
\leq\frac{C''}{n^8}(n^5\eta_n^6+n^6\eta_n^4),\end{multline*}
which is summable.

 Therefore, we conclude that 
 \begin{equation}\label{Dep}\limsup_n D(F^{B_n},F^{(1/\sqrt N)(\widetilde D_n\circ\widehat X_n)(\widetilde D_n\circ\widehat X_n)^*})\leq\epsilon\quad a.s.\end{equation}


 From this point, for ease of notation, we will denote the augmented matrix also by $B_n=(1/\sqrt N)(D_n\circ X_n)(D_n\circ X_n)^*$, 
 where we will assume that $\xn_{jk}$, the entries of $X_n$ satisfy
\begin{enumerate}
\item For each $n$ they are independent.
\item $\EE(\xn_{jk})=0$, $\EE|\xn_{jk}|^2=1$.
\item $|\xn_{jk}|\leq\eta_n\sqrt n$,\hfill\break
where the present $\eta_n$ are double the original $\eta_n$,
\end{enumerate}
and that the elements of $D_n$ are nonnegative and bounded by $d=d_{\epsilon}$.   Keep in mind the difference between the 
two $B_n$'s, along with the bound \eqref{Dep}.
 
Let $x_k$ denote the $k-$th column of $X_n$.   Let $D_k$ denote the $n\times n$ diagonal matrix consisting of the entries in the $k-$th column of $D_n$.   Then we can write
$$B_n=\frac1N\sum_{k\leq N}D_kx_kx_k^*D_k.$$
Let $\EE_0(\cdot)$ denote expectation and $\EE_k(\cdot)$ denote conditional expectation with respect to the $\sigma$-field generated by $x_1,\ldots,x_k$.  Let for $k\leq N$
$$B_{(k)}=\frac1N\sum_{j\neq k}D_jx_jx_j^*D_j.$$
Let $m_n(z)$, $z=x+iv$, $v>0$ denote the Stieltjes transform of the eigenvalues of $B_n$.   We have
$$m_n(z)=\frac1n\tr(B_n-zI)^{-1}.$$
The first step is to show
\begin{equation}
  \label{1.4}
 m_n(z)-\EE m_n(z)\to0\quad\text{a.s. as $n\to\infty$}.
\end{equation}
We have
\begin{multline*} m_n(z)-\EE m_n(z)=\frac1n\sum_{k=1}^N[\EE_k\tr(B_n-zI)^{-1}-\EE_{k-1}\tr(B_n-zI)^{-1}]
\hfill\\ \hfill=\frac1n\sum_{k=1}^N[\EE_k(\tr(B_n-zI)^{-1}-\tr(B_{(k)}-zI)^{-1})-\EE_{k-1}(\tr(B_n-zI)^{-1}-\tr(B_{(k)}-zI)^{-1})]\hfill\\ \hfill=\frac1n\sum_{k=1}^N\gamma_k.\end{multline*}
By Lemma \ref{L.1.8} we have each $|\gamma_k|\leq 2/v$.   Since the $\gamma_k$ form a martingale difference sequence, we have by Lemma \ref{L.1.6}
$$\EE|m_n(z)-\EE m_n(z)|^4\leq \frac{C_p}{n^4}\EE\left(\sum_{k=1}^N|\gamma_k|^2\right)^2\leq\frac{4C_pN^2}{v^4n^4},$$
which is summable.   Therefore we have   \eqref{1.4}.

We turn to $\EE m_n(z)$.  Let
$$F=\frac1N\sum_{k=1}^N\frac1{1+\frac nNe_k}D_k^2\quad\text{where}\quad e_k=\frac1n\EE\tr D_k(B_{(k)}-zI)^{-1}D_k.$$
Write
$$B_n-zI-(F-zI)=\frac1N\sum_{k=1}^ND_kx_kx_k^*D_k-F.$$
Taking inverses and using Lemma \ref{L.1.9} we get
\begin{multline}\label{1.5} (F-zI)^{-1}-(B_n-zI)^{-1}\hfill\\ \hfill=\frac1N\sum_{k=1}^N\biggl[\frac1{1+\frac1Nx_k^*D_k(B_{(k)}-zI)^{-1}D_kx_k}(F-zI)^{-1}D_kx_kx_k^*D_k(B_{(k)}-zI)^{-1}\hfill\\ \hfill
-(F-zI)^{-1}F(B_n-zI)^{-1}\biggr]
\end{multline} 
Taking the trace, and dividing by $n$, we get
\begin{multline*}\frac1n\tr(F-zI)^{-1}-m_n(z)=\frac1N\sum_{k=1}^N\biggl[\frac{\frac1nx_k^*D_k(B_{(k)}-zI)^{-1}(F-zI)^{-1}D_kx_k}{{1+\frac1Nx_k^*D_k(B_{(k)}-zI)^{-1}D_kx_k}}\hfill\\ \hfill
-\tr\frac1n(F-zI)^{-1}F(B_n-zI)^{-1}\biggr]\end{multline*}
$$=\frac1N\sum_{k=1}^N\left[\frac{\frac1nx_k^*D_k(B_{(k)}-zI)^{-1}(F-zI)^{-1}D_kx_k}{{1+\frac nN\frac1nx_k^*D_k(B_{(k)}-zI)^{-1}D_kx_k}}-\frac{\frac1n\tr(F-zI)^{-1}D_k^2(B_n-zI)^{-1}}{1+\frac nNe_k}\right]$$
$$=\frac1N\sum_{k=1}^N\alpha_k+\beta_k+\gamma_k,$$
where
$$\alpha_k=\frac{\frac1nx_k^*D_k(B_{(k)}-zI)^{-1}(F-zI)^{-1}D_kx_k}{{1+\frac nN\frac1nx_k^*D_k(B_{(k)}-zI)^{-1}D_kx_k}}-\frac{\frac1nx_k^*D_k(B_{(k)}-zI)^{-1}(F-zI)^{-1}D_kx_k}{{1+\frac nNe_k}},$$
$$\beta_k=\frac{\frac1nx_k^*D_k(B_{(k)}-zI)^{-1}(F-zI)^{-1}D_kx_k}{{1+\frac nNe_k}}-\frac{\frac1n\tr D_k(B_{(k)}-zI)^{-1}(F-zI)^{-1}D_k}{{1+\frac nNe_k}},$$
and
$$\gamma_k=\frac{\frac1n\tr D_k(B_{(k)}-zI)^{-1}(F-zI)^{-1}D_k}{{1+\frac nNe_k}}-\frac{\frac1n\tr(F-zI)^{-1}D_k^2(B_n-zI)^{-1}}{1+\frac nNe_k}.$$
Notice that the spectral norms of $(B_n-z)^{-1}$ and $(B_{(k)}-zI)^{-1}$ are bounded by $1/v$.  Also notice that $\Im e_k>0$ so that $\Im F<0$.  This implies $\|(F-zI)^{-1}\|<1/v$.

We have $\EE\beta_k=0$.   From Lemmas \ref{L.1.7} and \ref{L.1.8} we have
$$|\gamma_k|\leq\frac1n\frac{d^2|z|}{v^3}\to0.$$
as $n\to\infty$.

From Lemmas \ref{L.1.7}, \ref{L.1.10}, and the Cauchy-Schwarz inequality we have
\begin{multline*}|\EE\alpha_k|\leq\frac{n|z|^2}{Nv^2}((1/n^2)\EE|x_k^*D_k(B_{(k)}-zI)^{-1}(F-zI)^{-1}D_kx_k|^2)^{1/2}\hfill\\ \hfill((1/n^2)(\EE|x_k^*D_k(B_{(k)}-zI)^{-1}D_kx_k-\EE\tr D_k(B_{(k)}-zI)^{-1}D_k|^2)^{1/2}.\end{multline*}
We have
\begin{multline}
\label{2.7}
(1/n^2)\EE|x_k^*D_k(B_{(k)}-zI)^{-1}(F-zI)^{-1}D_kx_k|^2\leq\frac{d_n^4}{v^4n^2}\EE\|x_k\|^4\hfill\\ \hfill=\frac{d^4}{v^4n^2}\left(\sum_{j=1}^n\EE |x_{jk}|^4+n(n-1)\right)\leq\frac{d^4}{v^4n^2}(\eta_n^2n^2+n(n-1))=O(1)\end{multline}
We have
\begin{multline*}((1/n^2)(\EE|x_k^*D_k(B_{(k)}-zI)^{-1}D_kx_k-\EE\tr D_k(B_{(k)}-zI)^{-1}D_k|^2\hfill\\ \hfill\leq(2/n^2)(\EE|x_k^*D_k(B_{(k)}-zI)^{-1}D_kx_k-\tr D_k(B_{(k)}-zI)^{-1}D_k|^2\hfill\\ \hfill+(2/n^2)\EE|\tr D_k(B_{(k)}-zI)^{-1}D_k-\EE\tr D_k(B_{(k)}-zI)^{-1}D_k|^2.\end{multline*}
Using Lemma \ref{L.1.10} we have
\begin{multline}\label{2.8}(1/n^2)(\EE|x_k^*D_k(B_{(k)}-zI)^{-1}D_kx_k-\tr D_k(B_{(k)}-zI)^{-1}D_k|^2\hfill\\ \hfill\leq (C/n^2)(\eta_n^2n)\EE\tr D_k(B_{(k)}-zI)^{-1}D_k^2(B_{(k)}-\bar zI)^{-1}D_k\hfill\\ \hfill
\leq(C/n^2)(\eta_n^2n)n\EE\|D_k(B_{(k)}-zI)^{-1}D_k^2(B_{(k)}-\bar zI)^{-1}D_k\|\leq C\eta_n^2\to0\end{multline}
as $n\to\infty$.
As before, we can write $(1/n)(\tr D_k(B_{(k)}-zI)^{-1}D_k-\EE\tr D_k(B_{(k)}-zI)^{-1}D_k)$ as a sum of martingale differences, where now each term is bounded in absolute value by $2d^2/v$.  Using Lemma \ref{L.1.6} we find then, that
$$(1/n^2)\EE|\tr D_k(B_{(k)}-zI)^{-1}D_k-\EE\tr D_k(B_{(k)}-zI)^{-1}D_k|^2=O(1/n).$$
We conclude that
$$\EE m_n(z)-(1/n)\tr(F-zI)^{-1}\to0,$$
as $n\to\infty.$

In   \eqref{1.5}, multiply both sides on the left and right by $D_j$.  Taking the trace and dividing by $n$ we get
\begin{multline*}\frac1n\tr D_j(F-zI)^{-1}D_j-\frac1n\tr D_j(B_n-zI)^{-1}D_j\hfill\\ \hfill
=\frac1N\sum_{k=1}^N\left[\frac{\frac1nx_k^*D_k(B_{(k)}-zI)^{-1}D_j^2(F-zI)^{-1}D_kx_k}{{1+\frac nN\frac1nx_k^*D_k(B_{(k)}-zI)^{-1}D_kx_k}}\right.\hfill\\ \hfill
\left.-\frac{\frac1n\tr  D_j(F-zI)^{-1}D_k^2(B_n-zI)^{-1}D_j}{1+\frac nNe_k}\right].\end{multline*}
In approaching this sum in the same way as before we can conclude that
$$\EE (1/n)D_j(B_n-zI)^{-1}D_j-(1/n)\tr D_j(F-zI)^{-1}D_j\to0$$
as $n\to\infty$. Using Lemma \ref{L.1.8} we have
$$e_j(z)-(1/n)\tr D_j(F-zI)^{-1}D_j\to0$$
as $n\to\infty$.

Thus we have, expressed in terms of the entries of $D_n$
\begin{equation}\label{mlimit}m_n(z)-\frac1n{   \sum_{i=1}^n}\frac1{ \frac1N\sum_{k=1}^N\frac{d_{ik}^2}{1+\frac nNe_k(z)}-z}\to0\end{equation}
a.s. as $n\to\infty$, where the $e_j(z)$ satisfy
\begin{equation}
  \label{elimit}
  e_j(z)-\frac1n\sum_{i=1}^n\frac{d_{ij}^2}{ \frac1N\sum_{k=1}^N\frac{d_{ik}^2}{1+\frac nNe_k(z)}-z}\to0
\end{equation}
as $n\to\infty$.

\medskip



We shall prove that for fixed $n$ and $N$ there exist $e_1^0(z),\dots,e_N^0(z)\in\Bbb C^+$ that satisfies \eqref{Ge0def}.
Consider the matrix $B_{n\ell}=(1/(N\ell))(D^{(\ell)}\circ X^{(\ell)})(D^{(\ell)}\circ X^{(\ell)})^*$ where $X^{(\ell)}$ now is $n\ell\times N\ell$ containing i.i.d bounded standardized variables, and $D^{(\ell)}$ is $n\ell\times N\ell$ containing $\ell^2$ copies of $D_n$.
 Then  \eqref{elimit}  holds a.s. as $\ell\to\infty$. Notice that for any positive integer $k\leq N$ columns $k,N+k,\ldots,(\ell-1)N+k$ of $D^{(\ell)}$ are identical.  Using Lemma
 \ref{L.1.8} we see that
$$|e_{m_1N+k}(z)-e_{m_2N+k}(z)|\leq\frac{2d^2}{n\ell v}\quad 0\leq m_1,m_2\leq\ell-1.$$

Consider one realization for which  \eqref{elimit}  holds.  Since the $e_k(z)$'s are bounded by $d^2/v$, we can find a subsequence of $\{\ell\}$ such that $e_1(z),\ldots,e_N(z)$ converge to $e^0_1(z),\dots,e^0_N(z)$ on this subsequence.  We have then for $1\leq i\leq n\ell$
$$\frac1{N\ell}\sum_{k=1}^{N\ell}\frac{d^{(\ell)2}_{ik}}{1+\frac nNe_k(z)}=\frac1N\sum_{k=1}^N\frac1{\ell}\sum_{m=1}^{\ell}\frac{d_{ik}^2}{1+\frac nNe_{mN+k}(z)}
\to\frac1N\sum_{k=1}^N\frac{d^2_{ik}}{1+\frac nNe^0_k(z)},$$
on this subsequence.   Since the $e_k(z)$ have positive imaginary part, the above limit has nonpositive imaginary part.  Therefore the right side of \eqref{Ge0def} has positive imaginary part.

Fix $j\leq N$. Using the fact that the $d_{ij}^{(\ell)}$'s repeat as $i$ ranges from 1 to $n\ell$, from \eqref{elimit} we get  
$$e_j(z)-\frac1{n\ell}\sum_{i=1}^{n\ell}\frac{{d_{ij}^{(\ell)}}^2}{\frac1{N\ell}\sum_{k=1}^{N\ell}\frac{d^{(\ell)2}_{ik}}{1+\frac nNe_k(z)}-z}=e_j(z)-\frac1n\sum_{i=1}^{n}\frac{d_{ij}^2}{\frac1{N\ell}\sum_{k=1}^{N\ell}\frac{d^{(\ell)2}_{ik}}{1+\frac nNe_k(z)}-z}\to0,$$
so we get \eqref{Ge0def} with each $e_j(z)\in\Bbb C_+$.


We will now show the uniqueness of the $e_k^0(z)$ in  \eqref{Ge0def} having positive imaginary parts.  Notice that from \eqref{Ge0def}  the $e_k^0(z)$ are also bounded in absolute value by 
$d^2/v$.   Let $e^0_{j,2}(z)$ denote the imaginary part of $e^0_j(z)$.  Then
$$e^0_{j,2}(z)=\frac1n\sum_{i=1}^nd_{ij}^2\frac{\frac1N\sum_{k=1}^Nd_{ik}^2\frac{\frac nNe^0_{k,2}(z)}{|1+\frac nNe^0_k(z)|^2}+v}{\left|\frac1N\sum_{\underline k=1}^N\frac{d_{i\underline k}^2}{1+\frac nNe^0_{\underline k}(z)}-z\right|^2}.$$
Let $C^0=\{c^0_{jk}\}$ $N\times N$ and $b^0=(b_1^0,\ldots,b_N^0)^T$, where
$$c^0_{jk}=\frac1{N^2}\sum_{i=1}^nd^2_{ij}d^2_{ik}\frac{\frac1{|1+\frac nNe_k^0(z)|^2}}{\left|\frac1N\sum_{\underline k=1}^N\frac{d_{i\underline k}^2}{1+\frac nNe^0_{\underline k}(z)}-z\right|^2}$$
and
$$b^0_j=\frac1{n}\sum_{i=1}^n\frac{d^2_{ij}}{\left|\frac1N\sum_{\underline k=1}^N\frac{d_{i\underline k}^2}{1+\frac nNe^0_{\underline k}(z)}-z\right|^2}.$$
Let $e^0_2=(e^0_{1,2}(z),\ldots,e^0_{N,2}(z))^T$.  Then
\begin{equation}
  \label{1.8} e^0_2=C^0e^0_2+vb^0.
\end{equation}


Notice because of the nonzero assumption on any column of $D_n$, equation   \eqref{1.8}  has all components of $e_2^0(z)$ and $b^0$ positive 
 

Suppose $\rho(C^0)$, the spectral radius of $C^0$, is greater than or equal to 1.  Then by Lemma \ref{L.1.11} $\rho(C^0)$ is an eigenvalue of $C^0$ with left eigenvector $y^T$ containing nonnegative entries, and necessarily, at least one entry is positive.
Multiplying on the left of both sides of   \eqref{1.8}  by $y^T$ we get
\begin{equation}\label{ye2}y^Te_2^0=\rho(C^0)y^Te_2^0+vy^Tb^0.\end{equation}
Since both $y^Te_2^0$ and $y^Tb^0$ are positive, we arrive at a contradiction.  Therefore
\begin{equation}
  \label{1.9}
 \rho(C^0)<1.
\end{equation}
Suppose $\underline e_1^0(z),\ldots,\underline e_1^0(z)$ also satisfy   \eqref{Ge0def}.   Let $\underline e_2^0$ and $\underline C^0=\{\underline c_{jk}^0\}$ denote the quantities analgous to $e^0_2$ and $C^0$.  Then for nonzero $e_j^0(z)$ and $\underline e_j^0$(z)
we have
$$e_j^0-\underline e_j^0=\sum_{k=1}^Na_{jk}(e_k^0(z)-\underline e_k^0(z)),$$
where
\begin{multline}\label{Adef}a_{jk}=\frac1{N^2}\frac1{(1+\frac nNe^0_k(z))(1+\frac nN\underline e^0_k(z))}\hfill\\ \hfill
\times \sum_{i=1}^n\frac{d_{ij}^2d_{ik}^2}{\left(\frac1N\sum_{\underline k=1}^N\frac{d_{i\underline k}^2}{1+\frac nNe_{\underline k}^0(z)}-z\right)\left(\frac1N\sum_{\underline k=1}^N\frac{d_{i\underline k}^2}{1+\frac nN\underline e_{\underline k}^0(z)}-z\right)}.\end{multline}
Therefore, with $e^0=(e_1^0(z),\ldots,e_N^0(z))^T$ and  $\underline e^0=(\underline e_1^0(z),\ldots,\underline e_N^0(z))^T$ 
we have
\begin{equation}\label{e0diff}e^0-\underline e^0=A(e^0-\underline e^0),\end{equation}
where $A=(a_{jk})$.    If $e^0\neq\underline e^0$ then $A$ has an eigenvalue equal to 1.   However applying Cauchy-Schwarz we see that
$$|a_{jk}|\leq {c_{jk}^0}^{\frac12}{\underline c^0_{jk}}^{\frac12},$$
and applying Lemmas \ref{L.1.12} and \ref{L.1.13}
$$\rho(A)\leq\rho({c_{jk}^0}^{\frac12}{\underline c^0_{jk}}^{\frac12})\leq(\rho(C^0))^{\frac12}(\rho(\underline C^0))^{\frac12}<1,$$
by   \eqref{1.9}, a contradiction.  Therefore we have $e^0=\underline e^0$.

Using the last part of the above argument, we will show that $e^0(z)$ is a continuous function of $z\in\Bbb C^+$.  
Let $\{z_n\}$ and $\{z'_n\}$ be two sequences in $\Bbb C^+$ each converging to $z\in\Bbb C^=$.  Let $e^0(z)=(e_1^0(z),\ldots,e_N^0(z))^T$.  Then $e^0(z_n)$ and $e^0(z'_n)$ satisfy
\begin{equation}\label{e0ndiff}e^0(z_n)-e^0(z'_n)=A(z_n,z'_n)(e^0(z_n)-e^0(z'_n))+(z_n-z'_n)b(z_n,z'_n),\end{equation}
where 
\begin{multline*} A(z_n,z'_n)_{jk}\hfill\\=\frac1{N^2}\frac1{(1+\frac nNe^0_k(z_n))(1+\frac nN e^0_k(z'_n))}\hfill\\ \hfill\sum_{i=1}^n\frac{d_{ij}^2d_{ik}^2}{\left(\frac1N\sum_{\underline k=1}^N\frac{d_{i\underline k}^2}{1+\frac nNe^0_{\underline k}(z_n)}-z_n\right)\left(\frac1N\sum_{\underline k=1}^N\frac{d_{i\underline k}^2}{1+\frac nNe_{\underline k}^0(z'_n)}-z'_n\right)}.\end{multline*}
and
$$b(z_n,z'_n)_j=\frac1n\sum_{i=1}^n\frac{d_{ij}^2}{\left(\frac1N\sum_{\underline k=1}^N\frac{d_{i\underline k}^2}{1+\frac nNe^0_{\underline k}(z_n)}-z_n\right)\left(\frac1N\sum_{\underline k=1}^N\frac{d_{i\underline k}^2}{1+\frac nNe_{\underline k}^0(z'_n)}-z'_n\right)}.$$
Assume $e^0(z_n)\to e^0$ and $e^0(z'_n)\to\underline e^0$ as $n\to\infty$.  Since all quantities are bounded we get as $n\to\infty$ \eqref{e0diff},
and therefore $e^0=\underline e^0$, which proves continuity of $e^0(z)$, $z\in\Bbb C^+$.  

We claim that each $e_k(z)$ is the Stieltjes transform of a measure with total mass $(1/n)\tr D_k^2$.  Indeed, for all $z$ with imaginary part $\ge v>0$ any difference quotient
\begin{multline*}\frac{\frac1n\tr D_k(B_{(k)}-z_1I)^{-1}D_k-\frac1n\tr D_k(B_{(k)}-z_2I)^{-1}D_k}{z_1-z_2}\hfill\\ \hfill=\frac1n
\tr D_k(B_{(k)}-z_1I)^{-1}(B_{(k)}-z_2I)^{-1}D_k\end{multline*}
is uniformly bounded.  Therefore, by the dominated convergence theorem $e_k$ is differentiable for all $z\in\Bbb C^+$, and hence is analytic  on $\Bbb C^+$.  Moreover it is clear that $e_k$ maps $\Bbb C^+$ into $\Bbb C^+$ and, again by the dominated convergence theorem,  the limit of $ze_k(z)$ as $z\to\infty$ is $-(1/n)\tr D_k^2$.  Therefore by Lemma \ref{stielt} the claim is proven.

Now, since, from the above existence argument, for each $n$, $e_k^0$ is the limit of a sequence of $e_k$'s we have, using the Helly selection theorem, $e_k^0$ is itself the Stieltjes transform of a measure with total mass $(1/n)\tr D_k^2$.  Therefore the function $G_n(z)$ is analytic in $\Bbb C^+$, maps $\Bbb C^+$ to $\Bbb C^+$, and as $z\to\infty$, $zG_n(z)\to-1$.  Therefore, by Lemma \ref{stielt} $G_n(z)$ is the Stieltjes transform of a probability measure, $F_n^0$, on $\Bbb R$.

 We will proceed to show that the $e_k(z)$ in \eqref{mlimit} can be replaced with the $e_k^0(z)$, that is, we will show that
 \begin{equation}
 \label{me0limit}
 m_n(z)-G_n(z)=m_n(z)-\frac1n{\sum_{i=1}^n}\frac1{ \frac1N\sum_{k=1}^N\frac{d_{ik}^2}{1+\frac nNe_k^0(z)}-z}\to0
 \end{equation}
almost surely.  We have

\begin{equation}
\label{ee0diff}
\frac1n{   \sum_{i=1}^n}\frac1{ \frac1N\sum_{k=1}^N\frac{d_{ik}^2}{1+\frac nNe_k(z)}-z}-\frac1n{   \sum_{i=1}^n}\frac1{ \frac1N\sum_{k=1}^N\frac{d_{ik}^2}{1+\frac nNe_k^0(z)}-z}\end{equation}
\begin{multline*}=\frac1{N^2}\sum_{k=1}^N\frac{(e_k(z)-e_k^0(z))}{(1+\frac nNe_k(z))(1+\frac nNe_k^0(z))}\hfill\\ \hfill\times\sum_{i=1}^n\frac{d_{ik}^2}
{\left(\frac1N\sum_{\underline k=1}^N\frac{d_{i\underline k}^2}{1+\frac nNe_{\underline k}(z)}-z\right)\left(\frac1N\sum_{\underline k=1}^N\frac{d_{i\underline k}^2}{1+\frac nNe_{\underline k}^0(z)}-z\right)}.\end{multline*}

   Returning now to
$C^0$, $b^0$, and $e_2^0$, we see that the entries of $b^0$ are uniformly bounded.  Let $\bfone_m$ denote the $m$ dimensional vector containing all one's, and let ``$\leq$" between two vectors denote entry wise $\leq$.   Using Lemma \ref{L.1.7}
we see that
\begin{equation}\label{C0b0rel}
C^0\bfone_N\leq\left(\frac nN\frac{|z|^2}{v^2}\max_i\frac1N\sum_{\underline k=1}^Nd_{i\underline k}^2\right)b^0.
\end{equation}
Since the entries of $e_2^0$ are uniformly bounded we see from \eqref{1.8} that
\begin{equation}\label{e2bbd}
e_2^0\leq k_1b^0.
\end{equation}
From \eqref{ye2}, where $y$ is the nonnegative eigenvector of $C^0$ associated with $\rho(C^0)$ we get
\begin{equation}\label{rhoC0bd}
1-\rho(C^0)=\frac{vy^Tb^0}{y^Te_2^0}\ge k,\end{equation}
where necessarily $k\in(0,1).$
Let $\omega\in\Bbb R^n$ have for its $j$-th entry the imaginary part of the lefthand side of \eqref{elimit}.   From condition \eqref{1.3} we see that
$\eta_n$ approaches zero more slowly than $1/n$.   Thus, from the derivation of \eqref{elimit} it is clear that
\begin{equation}\label{omegabd}|\omega|\leq k_1\eta_n\bfone_N.\end{equation}

Let $e_{j,2}(z)$ denote the imaginary part of $e_j(z)$ and $e_2=(e_{1,2},\ldots,e_{N,2})^T$.   Similar to $C^0$ and $b^0$ we have
\begin{equation}\label{Cebomega}
e_2=Ce_2+vb+\omega,\end{equation}
where
$$c_{jk}=\frac1{N^2}\sum_{i=1}^nd_{ij}^2d_{ik}^2\frac{\frac1{|1+\frac nNe_k(z)|^2}}{\left|\frac1N\sum_{\underline k=1}^N\frac{d_{i\underline k}^2}{1+\frac nNe_{\underline k}(z)}-z\right|^2}$$
and
$$b_j=\frac1{n}\sum_{i=1}^n\frac{d_{ij}^2}{\left|\frac1N\sum_{\underline k=1}^N\frac{d_{i\underline k}^2}{1+\frac nNe_{\underline k}(z)}-z\right|^2}.$$
We will show
\begin{equation}\label{specradcon}\limsup_{n\to\infty}\rho(C)\leq1.\end{equation}
Rather than work with $C$ we consider $C_s\equiv E^{-1}CE$ where $E=\text{diag}(|1+\frac nNe_1(z)|,\ldots,|1+\frac nNe_N(z)|)$. Notice $C_s$ is a symmetric matrix and its eigenvalues are the same as $C$.  From Lemma 1.7 and the fact that $|1+\frac nNe_j(z)|\leq 1 + \frac 1{Nv}\sum_{i=1}^n{d_{ij}^2}$, we see that the diagonal entries of $E$ and $E^{-1}$ are uniformly bounded.   Let $f=E^{-1}e_2$, $\underline b=E^{-1}vb$, and $\underline\omega=E^{-1}\omega$.  Then from \eqref{Cebomega} we have
\begin{equation}\label{Csfbpmega}f=C_sf+\underline b+\underline\omega.\end{equation}
We have the entries of $\underline b$ uniformly bounded,
$$|\underline\omega|\leq k_2\eta_n\bfone_N$$ and, similar to \eqref{C0b0rel}, we have
\begin{equation}\label{Csbrel}C_s\bfone_N\leq k_3\underline b.\end{equation}
Consider those entries of $f_i$ for which $\underline b_i+\underline\omega_i>0$ and those entries for which  $\underline b_i+\underline\omega_i$ are negative.   Rearrange the coordinates so that the first $\ell$ entries satisfy the former, the remaining the latter and put $C_s$ into corresponding block form:
$$C_s=\begin{pmatrix} C_{11}&C_{12}\\ C_{21}&C_{22}\end{pmatrix}.$$
Splitting $f$, $\underline b$ and $\underline\omega$ into appropriate parts we have
\begin{equation}\label{fCbo1}f_1=C_{11}f_1+C_{12}+\underline b_1+\underline\omega_1\end{equation}
\begin{equation}\label{fCbo2}f_2=C_{21}f _1+C_{22}f_2+\underline b_2+\underline\omega_2,\end{equation}
where $f=(f_1^T,f_2^T)^T$, etc. Applying the left nonnegative eigenvector of $C_{11}$ corresponding to $\rho(C_{11})$ to both sides of \eqref{fCbo1} we see that $\rho(C_{11})<1$.   There exists $k_4$ for which
\begin{multline}\label{etabd}C_{11}\bfone_{\ell}\leq k_4\bfone_{\ell},\ C_{12}\bfone_{N-\ell}\leq k_4\bfone_{\ell},\ \underline b_2\leq k_4\eta_n\bfone_{N-\ell},\  C_{21}\bfone_{\ell}\leq k_4\eta_n\bfone_{N-\ell},\hfill\\ \hfill\text{ and }C_{22}\bfone_{N-\ell}\leq k_4\eta_n\bfone_{N-\ell}.\end{multline}
From Lemma 1.17 we see that $\rho(C_{22})\leq k_4\eta_n$, which is less than 1/2 for all $n$ large.   Notice then for these $n$ (which we assume we consider from this point on) that if $\ell=0$ or $N$ then $\rho(C)=\rho(C_s)<1$ and we would be done.  So we assume $1<\ell<N$.

Let $x=(x_1^Tx_2^T)^T$ be a nonnegative eigenvector of of $C_s$ corresponding to $\rho(C_s)$.  Then we have
$$\rho(C_s)x_1=C_{11}x_1+C_{12}x_2$$
$$\rho(C_s)x_2=C_{21}x_1+C_{22}x_2.$$
When $\rho(C_s)>1$ we have $\rho(C_s)I-C_{22}$ invertible (smallest eigenvalue $\ge 1/2$ and $\|(\rho(C_s)I-C_{22})^{-1}\|\leq2$.
From the second identity we have $x_2=(\rho(C_s)-C_{22})^{-1}C_{21}x_1$ and plugging this into the first we find
\begin{equation}\label{x1eq}\rho(C_s)x_1=(C_{11}+C_{12}(\rho(C_s)I-C_{22})^{-1}C_{21})x_1.\end{equation}
We have when $\rho(C_s)\ge1$

\begin{multline*}(\rho(C_s)I-C_{22})^{-1}\bfone_{N-\ell}=\frac1{\rho(C_s)}\sum_{m=0}^{\infty}\frac{C_{22}^m}{\rho(C_s)^m}\bfone_{N-\ell}\hfill\\ \hfill\leq\frac1{\rho(C_s)}\sum_{m=1}^{\infty}\frac1{(2\rho(C_s))^m}\bfone_{N-\ell}=\frac1{\rho(C_s)}\frac1{1-\frac1{2\rho(C_s)}}\bfone_{N-\ell}=\frac1{\rho(C_s)-\frac12}\bfone_{N-\ell}.\end{multline*}
From \eqref{etabd} we have
$$C_{12}(\rho(C_s)I-C_{22})^{-1}C_{21}\bfone_{\ell}\leq\frac{\eta_nk_4^2}{\rho(C_s)-\frac12}\bfone_{\ell}.$$
Therefore, for all $n$ large, when $\rho(C_s)\ge1$
\begin{equation}\label{c12bd}\rho(C_{12}(\rho(C_s)I-C_{22})^{-1}C_{21})=\vertiii{C_{12}(\rho(C_s)I-C_{22})^{-1}C_{21}}_2\leq k_5\eta_n\end{equation}
If $x_1=0$, then $\rho(C_s)=\rho(C_{22})\leq1/2$.   Therefore, when $\rho(C_s)\ge1$ we have $x_1\neq0$ and so applying the spectral norm on \eqref{x1eq} we have
$$\rho(C_s)\leq\rho(C_{11})+k_5\eta_n,$$
and so \eqref{specradcon} holds.

Let $e=(e_1(z),\ldots,e_N(z))^T$. Let now $\omega\in\Bbb C^n$ have for its entries the lefthand side of  \eqref{elimit}.  We may assume
\begin{equation}\label{obd}|\omega|\leq k_4\eta_n\bfone_N.\end{equation}
From \eqref{Ge0def} and \eqref{elimit}, we have
$$e-e^0=A(e-e^0) + \omega,$$
where $A=(a_{jk})$ with
\begin{multline*}a_{jk}=\frac1{N^2}\frac1{(1+\frac nNe_k(z))(1+\frac nN e^0_k(z))}\hfill\\ \hfill\sum_{i=1}^n\frac{d_{ij}^2d_{ik}^2}{\left(\frac1N\sum_{\underline k=1}^N\frac{d_{i\underline k}^2}{1+\frac nNe_{\underline k}(z)}-z\right)\left(\frac1N\sum_{\underline k=1}^N\frac{d_{i\underline k}^2}{1+\frac nNe_{\underline k}^0(z)}-z\right)}.\end{multline*}

Using the same arguments seen earlier, we have from  \eqref{rhoC0bd}, \eqref{specradcon} for all $n$ large the existence of
a $k_6\in(0,1)$ such that
$$\rho(A)\leq k_6.$$
Therefore I-A is invertible and we get
\begin{equation}\label{IAinv}
e-e^0=(I-A)^{-1}\omega.\end{equation}



Let $D_2=D_n\circ D_n$, $F^0=\text{diag}(1+\frac nNe_1^0(z),\ldots,1+\frac nNe_N^0(z))$, $F=\text{diag}(1+\frac nNe_1(z),\ldots,1+\frac nNe_N(z))$, 
$$G^0=\text{diag}\left(\frac1{\frac1N\sum_{\underline k=1}^N\frac{d^2_{1\underline k}}{1+\frac nMe^0_{\underline k}(z)}-z},\ldots,\frac1{\frac1N\sum_{\underline k=1}^N\frac{d^2_{n\underline k}}{1+\frac nMe^0_{\underline k}(z)}-z}\right),
$$
and
$$G=\text{diag}\left(\frac1{\frac1N\sum_{\underline k=1}^N\frac{d^2_{1\underline k}}{1+\frac nMe_{\underline k}(z)}-z},\ldots,\frac1{\frac1N\sum_{\underline k=1}^N\frac{d^2_{n\underline k}}{1+\frac nMe_{\underline k}(z)}-z}\right).
$$
Then we can write
$A=\frac1{N^2}D_2^TGG^0D_2{F^0}^{-1}F^{-1}$.  Writing $F^{-1}=E^{-1}\text{diag}(\omega_1,\ldots,\omega_N)$ where the $\omega_i$'s are on the unit circle in the complex plane, we see that  since $((1/N)GD_2E^{-1})^*(1/N)GD_2E^{-1}=C_s$, we have from Lemma \ref{supnormdiag} $\vertiii{(1/N)GD_2F^{-1}}_2=\rho^{1/2}(C)$.  Similarly we have $\vertiii{(1/N)G^0D_2{F^0}^{-1}}_2=\rho^{1/2}(C^0)$.   Therefore, for all $n$ large we have $k_7\in(0,1)$ such that
$$\vertiii{F^{-1}AF}_2\leq\vertiii{(1/N)GD_2F^{-1}}_2\vertiii{(1/N)G^0D_2{F^0}^{-1}}_2<k_7,$$
and so $(I-F^{-1}AF)^{-1}$ exists and is bounded in spectral norm for all $n$ large.  Therefore, using the fact that $F$ is bounded in spectral norm we have
$$e-e^0=F(I-F^{-1}AF)^{-1}F^{-1}\omega=H\omega,$$
where $H$ is bounded in spectral norm for all $n$ large.

We have then $\eqref{ee0diff}=\frac1N\bfone_N^TGG^0(1/N)D_2F^{-1}{F^0}^{-1}H\omega.$  It is straightforward to verify that the entries of the vector $\bfone_N^TGG^0(1/N)D_2F^{-1}{F^0}^{-1}$ are bounded for all large $n$.
Using the fact that $\|\bfone_N\|_2=\sqrt N$ along with  \eqref{obd} we conclude that

\begin{multline*}\left|\frac1N\bfone_N^TGG^0(1/N)D_2F^{-1}{F^0}^{-1}H\omega\right|\leq \frac1N\|\bfone_N^TGG^0(1/N)D_2F^{-1}{F^0}^{-1}\|_2\vertiii{H}_2\|\omega\|_2\hfill\\ \hfill\leq k_8\eta_n\to0\end{multline*}
as $n\to\infty$.  Therefore we get \eqref{me0limit}.

We proceed to complete the proof of the theorem.   With probability one, say on the set $A$, \eqref{me0limit} holds for a countably infinite collection $\{z_m\}$ of $z\in \Bbb C^+$ uniformly bounded away from the real axis having a cluster point.   
For any $\omega\in A$, and any vaguely converging subsequence of $F^{B_{n_j}}$, say to $F_{\mu}$ with $\mu$ a sub-probability measure, we have $m_{n_j}(z)\to m_{\mu}(z)
\equiv\int\frac1{x-z}dF_{\mu}$, $z\in\{z_m\}$.   
Necessarily $G_{n_j}(z)\to m_{\mu}(z)$,  $z\in\{z_m\}$. By  Lemma \ref{vaguelimit} we have $F_{n_j}^0$ converging vaguely to the distribution function of a measure, which, because of 
uniqueness of measures and their Stieltjes transforms, must necessarily be $\mu$.   Thus, $D(F^{B_{n_j}},F_{n_j}^0)\to0$ on this subsequence.   Since by the Helly selection theorem for an arbitrary subsequence, there exists a further
subsequence for which vague convergence holds, we must have 
$$D(F^{B_{n}},F_{n}^0)\to0,\quad \omega\in A.$$
Combining this result with \eqref{Dep} we have for any $\epsilon>0$ with probability one
$$\limsup_nD(F^{B_n},F_{n,\epsilon}^0)\leq\epsilon,$$
where $B_n$ is now the original matrix defined in \eqref{Bndef}
and $F_{n,\epsilon}$ is the distribution function of the probability measure having Stieltjes transform $G_{n,\epsilon}$, which is $G_n$ defined in terms of the truncated $d_{jk}^n$'s,
namely $d_{jk}^nI(d_{jk}^n\leq d_{\epsilon})$ with $d_{\epsilon}\ge M_{\epsilon}$.

For the proof of the corollary, we let $A$ be a set of probability one for which for each $\omega\in A$
$$\limsup_n D(F^{B_n},F_{n,1/m}^0)\leq 1/m,\quad m=1,2,\ldots.$$
For fixed $\omega\in A$, choose integers $N_2>N_1>0$ arbitrarily.  Choose integer $N_3>N_2$ for which $D(F^{B_n},F_{n,1/3}^0)\leq 2/3$ for all $n\ge N_3$, and recursively
choose $N_m>N_{m-1}$ for which $D(F^{B_n},F_{n,1/m}^0)\leq 2/m$ for all $n\ge N_m$.  We thus have the existence of $\{\epsilon_n\}$, such that $D(F^{B_n},F_{n,\epsilon_n}^0)\to 0$, an event which occurs with probability one.

We conclude with a way to compute $e^0$ associated with $D_n=(d_{ij})$.  We will show that there exists a neighborhood of $e^0$ for which the scheme
\begin{equation}
  \label{1.10}
  e_j^{\ell+1}= \frac1n\sum_{i=1}^n\frac{d_{ij}^2}{ \frac1N\sum_{k=1}^N\frac{d_{ik}^2}{1+\frac nNe^{\ell}_k(z)}-z}.
\end{equation}
converges to $e^0$.    Consider the matrix $A$ in \eqref{e0diff} with $\underline e^0$ replaced by $e$ and denote it by $A(e)$.  Then we have $\rho(A(e^0))<1$.   By Lemma \ref{L.1.16} we can find a vector norm $\|\cdot\|$ where its induced matrix norm $\vertiii{\cdot}$ satisfies $\vertiii{A(e^0)}\leq\alpha<1$ for some $\alpha$.   Then, for a given $\beta\in(\alpha,1)$,  by continuity we can find a $\|\cdot\|$ open ball $\Bbb B$ of $e^0$ such that $\vertiii{A(e)}\leq\beta$ for all $e\in\Bbb B$.  Therefore, writing $e^{\ell}=(e_1^{\ell}(z),\ldots,e_N^{\ell}(z))^T$, $e^{\ell+1}=(e_1^{\ell+1}(z),\ldots,e_N^{\ell+1}(z))^T$, if $e^{\ell}\in\Bbb B$ we have
$$e^0-e^{\ell+1}=A(e^{\ell})(e^0-e^{\ell}),$$
and $$\|e^0-e^{\ell+1}\|\leq\vertiii{A(e^{\ell})}\|e^0-e^{\ell}\|\leq\beta|\|e^0-e^{\ell}\|.$$
Therefore $e^{\ell+1}\in\Bbb B$ and we get convergence to $e^0$.

\subsection*{Acknowledgement} {Special thanks go to Zhidong Bai for his aid in deciding on the tightness-like condition on the entries of $D_n$.}

\end{document}